\documentclass[a4paper,12pt]{article}
\usepackage{cmap}					
\usepackage[T2A]{fontenc}			
\usepackage[utf8]{inputenc}			
\usepackage[english]{babel}	
\usepackage{indentfirst}            
\usepackage{amsmath,amsfonts,amssymb,amsthm,mathtools} 
\usepackage{icomma} 
\usepackage{mathrsfs} 
\DeclareMathAlphabet{\mathpzc}{OT1}{pzc}{m}{it} 

\usepackage{graphicx}  
\graphicspath{{images/}{images2/}}  
\usepackage{wrapfig} 
\usepackage{caption}
\usepackage{array,tabularx,tabulary,booktabs} 
\usepackage{longtable}  
\usepackage{multirow} 

\usepackage{etoolbox} 

\usepackage{extsizes} 
\usepackage{geometry} 
\usepackage{setspace} 

\usepackage{lastpage} 

\usepackage{soul} 

\usepackage{multicol} 
\usepackage{csquotes} 
\usepackage{float} 

\usepackage{tikz} 
\usetikzlibrary{arrows}
\usepackage{pgfplots}
\usepackage{pgfplotstable}
\usepackage{comment}

\usepackage{colortbl} 

\usepackage{eucal}
\usepackage[hidelinks]{hyperref}

\newcommand\restr[2]{{
  \left.\kern-\nulldelimiterspace 
  #1 
  \littletaller 
  \right|_{#2} 
  }}
  \newcommand{\littletaller}{\mathchoice{\vphantom{\big|}}{}{}{}}

\renewcommand{\ge}{\geqslant}
\renewcommand{\le}{\leqslant}
\newcommand{\eps}{\varepsilon}
\newcommand{\pf}{\hookrightarrow}

\newcommand{\norm}[1]{\left \lVert#1\right \rVert}

\renewcommand{\phi}{\varphi}

\renewcommand{\C}{\mathbb{C}}
\newcommand{\Z}{\mathbb{Z}}

\newcommand{\X}{\mathscr{X}}

\newcommand{\op}[1]{\operatorname{#1}}
\newcommand{\mc}[1]{\mathcal{#1}}

\usepackage{yhmath}
\newcommand{\cC}{\widehat{\C}}

{
\theoremstyle{definition}
\newtheorem{example}{Example}

\newtheorem{remark}{Remark}

\newtheorem{definition}{Definition}

}
\newtheorem{theorem}{Theorem}
\newtheorem{intro-theorem}{Theorem}
\newtheorem{proposition}{Proposition}
\newtheorem{corollary}{Corollary}
\newtheorem{lemma}{Lemma}

\usepackage{appendix}

\usepackage[
backend=biber,
style=alphabetic,
sorting=ynt
]{biblatex}
\title{The derivation problem for one type of bimodules over group algebras}
\author{Andronick Arutyunov \and Alexey Naianzin}

\begin{document}

\maketitle

\begin{abstract}
    The derivation problem is a familiar one concerning group algebras, particularly $L_1(G)$ and von Neumann algebras. In this paper, we study the Banach bimodule $\ell_p(G)$, which is generated by the $\ell_p$ norm over a specific class of groups with well-organized conjugacy classes. For this case, we will demonstrate that all $\ell_p(G)$ derivations are inner.
\end{abstract}

\section{Introduction}

This article is devoted to the study of derivations in bimodules over group rings equipped with various types of norms. The research mainly focuses on classes of groups that act in some way nicely on their conjugacy graph. By employing an approach proposed in earlier works, it is demonstrated that for this class of groups all derivations are inner.

Before providing the historical background and formulating the problem, let us introduce some notation.

\begin{definition}[\cite{Dales-en}, Definition 1.8.1]
Let $A$ be an algebra over $\mathbb{C}$, and $M$ be an $A$-bimodule. A \textit{derivation} in the algebra $A$ with values in $M$ is a linear map
\begin{equation}
    d\colon A \to M,
\end{equation}
such that for all $a,\, b \in A$ we have
\begin{equation}
    d(ab) = d(a)b + ad(b).
\end{equation}
\end{definition}

Let us consider an important example of derivations.

\begin{definition}
Given $x \in M$, the \textit{inner derivation} $D_x\colon A \to M$ is a derivation defined by the formula
\begin{equation}
    D_x(a) = xa - ax.
\end{equation}
\end{definition}

In appendix \ref{appendix:example-unbounded}, we provide an example of discontinuous inner derivation in  $l_2(G).$ While existence of unbounded inner derivations may be a known fact, we have not encountered any prior mention of it in the literature.

A classical question known as a derivation problem (also known as a Johnson's problem) is whether all continuous derivations are inner. See \cite{Dales-en,AruMis19-en} for more details. Typically the case of $A = \mathbb{C}[G]$ is considered. In terms of Hochschild cohomology, it means the triviality of the first cohomology group. 

This problem is stated in \cite{Dales-en} as follows: 
Let $G$ be a locally compact group. Is every derivation from $L^1(G)$ to $M(G)$ inner? Here, $M(G)$ denotes the space of complex-valued regular Borel measures on $G$.

In many special cases, B. Johnson provided answers to this problem. He investigated it as a suitable example for the theory of homology in Banach algebras. For instance, in \cite{JOHNSON-2001-en} he showed that for a connected Lie group $G$ all derivations from $L_1\left(G\right)$ to itself have the form:
\begin{equation}
    Da = a\mu - \mu a,
\end{equation}
where $\mu\in M(G)$. In the work \cite{Johnson-1969}, he proved that all derivations of the algebra $\ell_1(G)$ for a discrete group $G$ are inner. The original formulation of the problem was discussed in \cite{Vikto-2008-en} by V. Losert. 

The algebra of outer derivations in group rings without a norm turns out to be nontrivial. For instance any nontrivial central derivation is not inner (See \cite{Aru20-en}, Definition 3). The structure of this algebra was studied in \cite{AruMis19-en, AruMisSht16-en}, and its description in terms of the original group $G$ was provided. These works also developed a technique for describing derivations using characters. 



\subsection*{Main results}

In this article we introduce a class of groups that we will call BC-groups (see Definition \ref{def:bounded-conjugations}), this are groups such that $\sup_{g\in G}\op{diam}(g^{-1}Kg)<\infty$ for any finite $K\subset G$. That is, conjugations don't change the distances between elements very much. 

The main goal of this paper is to prove that for BC-groups $G$ all continuous derivations in $\ell_p(G),\, p\geq 1$ are inner. Here is the precise statement, the \hyperref[th:main]{proof} is given in the subsection \ref{subsec:Case of Multiple Components}.
\setcounter{intro-theorem}{1}

\begin{intro-theorem} Let $G$ be a \hyperref[def:bounded-conjugations]{BC-group} with uniformly bounded finite \hyperref[def:connected-component]{components}. Then every \hyperref[def:G-bounded-derivation]{$G$-bounded} derivation $d\colon \hyperlink{finite-combinations}{f_p(G)}\to \hyperlink{finite-combinations}{\ell_q(G)}$ is inner.
\setcounter{intro-theorem}{0}
\end{intro-theorem}

All necessary notation is introduced below.

\section{Key Definitions}

Hereinafter, we consider only finitely generated groups $G$, i.e., groups which can be presented as  $G = \langle  \X | R \rangle$, where $\X = \{x_1, \dots, x_n\}$ is a finite set of generators, and ${R = \{r_i | i \in I\}}$ is a set of relations.

Firstly, let us introduce some definitions and notations.

\begin{definition}
The \textit{ $\ell_p$ norm} on the group ring $\C[G]$ is a norm of the form
\begin{equation}
    \norm{\sum_{g\in G}\alpha(g)g}_p = \sqrt[p]{\sum_{g\in G}|\alpha(g)|^p}.
\end{equation}

The \textit{supremum} or $\textnormal{sup}$ \textit{norm} is a norm of the form
\begin{equation}
    \norm{\sum_{g\in G}\alpha(g)g}_\textnormal{s} = \sup_{g \in G}{|\alpha(g)|}.
\end{equation}
\end{definition}

We denote the group ring equipped with the $\ell_p$ or $\sup$ norm as $f_p(G)$ or \hypertarget{finite-combinations}{$f_{\infty}(G)$}, respectively. Let $\ell_p(G)$ and $c_0(G)$ be their completions. By $\cC[G]$ denote a completion in case of arbitrary norm.


Note that in both cases of the $\ell_p$ and $\sup$ norms, multiplication by an element of the group $G$ is a continuous operator from $\C[G]$ to $\C[G]$, so it extends to an operator from $\cC[G]$ to $\cC[G]$. Therefore $\cC[G]$ is a bimodule over $\C[G].$ 

Note that a continuous derivation $d\colon\left(\C[G], \norm{\cdot}_1\right) \to \left(\cC[G], \norm{\cdot}_2\right)$ can be extended to a continuous operator $\hat{d}\colon\left(\cC[G], \norm{\cdot}_1\right) \to \left(\cC[G], \norm{\cdot}_2\right)$.

The technique described in \cite{AruMisSht16-en} allows us to assign a character on a groupoid to each derivation and vice versa. Let us recall the main definitions. We will mostly follow the notation suggested in \cite{Aru20-en}.

\begin{definition} The \textit{groupoid of conjugacy action $\Gamma$} is a small category where the objects are the elements of $G$.
The set of morphisms $\operatorname{Hom}(\Gamma)$  consists of all possible pairs of elements $(u, v) \in G \times G$, here $(u, v) \in \operatorname{Hom}_{\Gamma}(v^{-1}u, uv^{-1}).$ 

Let $\phi=\left(u_1, v_1\right)$ and $\psi=\left(u_2, v_2\right)$ be composable, then their composition $\psi \circ \phi$ is defined by the formula
$$
\psi \circ \phi:=\left(v_2 u_1, v_2 v_1\right).
$$
\end{definition}


\begin{definition}
A map $\chi\colon \op{Hom}(\Gamma) \to \C$ is called \textit{a character} if for any two composable arrows  $\phi=\left(u_1, v_1\right)$, $\psi=\left(u_2, v_2\right)$ we have
$$
\chi(\psi \circ \phi) = \chi(\phi) + \chi(\psi).
$$
\end{definition}

\begin{definition}
\label{def:connected-component}
 By $\Gamma_{[u]}$ denote the connected component of $\Gamma$ containing an element $u\in G.$ In other words, $\Gamma_{[u]}$ is a subgroupoid of $\Gamma$ such that 
 $$\op{Obj}(\Gamma_{[u]}) = \{g \in G \mid \op{Hom}(u,g) \neq \emptyset\},\quad  \op{Hom}(\Gamma_{[u]}) = \{(u,v)\in \op{Hom}(\Gamma) \mid v^{-1}u, uv^{-1} \in \op{Obj}(\Gamma_{[u]}) \}.$$ 
\end{definition}

Let $\delta_h\colon \cC[G]\to \C$ denote a map defined by the formula
$\delta_h\left(\sum_{g \in G}\alpha(g)g\right) = \alpha(h).$

\begin{definition}
    Given a derivation $d\colon\left(\C[G], \norm{\cdot}_1\right) \to \left(\cC[G], \norm{\cdot}_2\right)$,  the \textit{assigned character} 
    is a function $\chi\colon \op{Hom}(\Gamma)\to \C$ defined by the formula 
    \begin{equation}
        \chi(h,g) = \delta_h \left(d(g)\right).
    \end{equation}
\end{definition}

\begin{proposition}
Assigned character is indeed a character and it satisfies the formula
\begin{equation}
    d(g) = \sum_{h\in G} \chi(h,g) h = g\left(\sum_{t \in G} \chi(gt, g) t \right), \quad \forall g \in G.
    \label{eq:d(g)-chi}
\end{equation}
\end{proposition}

\begin{proof}
It is easy to see that
\begin{equation}
    d(g) = \sum_{h\in G} \delta_h \left(d(g)\right) h = \sum_{h\in G} \chi(h,g) h.
\end{equation}
The equality
\begin{equation}
    \sum_{h\in G} \chi(h,g) h = g\left(\sum_{t \in G} \chi(gt, g) t \right)
\end{equation}
is obtained by changing the summation indices $h = gt.$ 
Verification that $\chi$ is a character is conducted similarly to the proof of Theorem 1 in \cite{AruMisSht16-en}. For the sake of clarity, it will be repeated here.

If the composition $(h_2,g_2) \circ (h_1, g_1)$ is defined, then
\begin{align}
    &h_1 g_1^{-1} = g_2^{-1} h_2, \label{eq:compose_morphisms}
    \\
    &(h_2,g_2) \circ (h_1, g_1) = (g_2 h_1, g_2 g_1).
\end{align}
Taking this into account, we have
\begin{align}
    &\chi(g_2 h_1, g_2 g_1) = \delta_{g_2 h_1} \left(d(g_2 g_1)\right) = \delta_{g_2 h_1} \left(d(g_2) g_1\right) + \delta_{g_2 h_1} \left(g_2 d(g_1)\right) = \nonumber
    \\
    &= \delta_{g_2 h_1 g_1^{-1}} \left(d(g_2)\right) + \delta_{ h_1} \left( d(g_1)\right) \stackrel{(\ref{eq:compose_morphisms})}{=} \delta_{h_2} \left(d(g_2)\right) + \delta_{ h_1} \left( d(g_1)\right) = \chi(h_2,g_2) + \chi(h_1,g_1).
\end{align}
This implies that $\chi$ is a character and completes the proof.
\end{proof}


Recall that a derivation is called quasi-inner if the character assigned to this derivation equals zero on all loops. A morphism $(u,v)$ is called a loop if $uv = vu.$ One can see that every inner derivation is quasi-inner. 

Let a norm on the group algebra be subordinate to the $\op{sup}$ norm, then all continuous derivations are quasi-inner. The proof is given in \cite{Aru-comb-view-en}.

\begin{definition}
A \textit{potential} of a character $\xi$ is a function $\phi \colon \op{Obj}(G) \to \C$ such that 
\begin{equation}
    \chi(h,g) = \phi(hg^{-1}) - \phi(g^{-1}h). \label{eq:character-potential}
\end{equation}
\end{definition}
Conversely, each potential induces a character defined by the formula above. 

One can show that a potential of a character $\chi$ exists iff $\chi$ is quasi-inner. Let $\phi$, $\phi'$ be two potentials, if $\psi = \phi - \phi'$ is constant on connected components function, then $\phi$ and $\phi'$ induce the same character.

Rewriting formula (\ref{eq:d(g)-chi}) in terms of potentials, we get
\begin{equation}
    d(g) = \sum_{h\in G} \left(\phi(hg^{-1}) - \phi(g^{-1}h)\right)h = \sum_{t\in G} \left(\phi(gtg^{-1})  - \phi(t)\right)gt.
    \label{eq:d(g)-phi}
\end{equation}

Consider a formal sum $a = \sum_{t\in G} \phi(t)t,$ then the formula can be rewritten as:
\begin{equation}
d(g) = \sum_{t\in G} \phi(t) \left(tg - gt \right) = [a,g].
\label{eq:formal-commutator}
\end{equation}

Let $D_x\colon \C[G] \to \ell_p(G)$ be an inner derivation, where $x = \sum_{g\in G} \alpha(g)g \in \ell_p(G).$ Then by the last formula we obtain that $\alpha\colon G\to \C$ is the potential of $D_x.$ Conversely, a quasi-inner derivation $d$ is inner if and only if we can find a potential $\phi$ such that $a$ is not only a formal linear combination, but also an element of $\left(\cC[G],\norm{\cdot}\right)$, which is equivalent to $\norm{a} < \infty.$

\begin{definition}
    Let $G$ be a group and $\X$ be a generating set of $G$. The \textit{conjugacy graph} $\op{sk}=\op{sk}(G, \X)$ is an edge-labeled directed graph constructed as follows:
\begin{itemize}
    \item  Each element $g$ of $G$ is assigned a vertex: the vertex set of $\Gamma$ is identified with $G$.
    \item For every $g \in G$ and $x \in \X \cup \X^{-1}$ there is a directed edge with label $x$ from the vertex corresponding to $g$ to the one corresponding to $xg x^{-1}$.
\end{itemize} 
\end{definition}


The distance between vertices is defined as the minimum number of edges in the paths from one vertex to another. Note that all balls in $\op{sk}(G)$ contain a finite number of elements, and distance is infinite if elements lie in distinct connected components. 

\begin{remark}
    The graph $\op{sk}(G,\X)$ can be embedded in the groupoid $\Gamma.$ The set of vertices of $\op{sk}(G)$ and objects of $\Gamma (G)$ just coincide. An edge with label $x$ connecting $g$ and $xgx^{-1}$ maps to the morphism $(xg,x)\in \op{Hom}(g,xgx^{-1}).$
\end{remark}

Let us look at an example of a conjugacy graph. 

\begin{example} Consider the Heisenberg group 
$$H_3(\Z) = \langle A_x, A_p, A_1 \mid [A_p, A_x] = A_1, [A_p, A_1] = E, [A_x, A_1] = E \rangle.$$ 
The graph $\op{sk}(A_p)$ is depicted below. We have $A_x A_p A_1^{k} A_x^{-1} = A_p A_1^{k-1}$, so the edges connecting distinct vertices are labeled by $A_x.$ Edges labeled by $A_p, \, A_1$ are loops. Edges labeled by $A_x^{-1}, \,A_p^{-1},\, A_{1}^{-1}$ are not shown.  To avoid cluttering the notation on the figures in the future, we will refrain from depicting cycles on the graphs. 

 \begin{figure}[H]
	\begin{center}
            \begin{tikzpicture}[font=\small,scale = 1, every node/.style={scale=0.6}]
            \foreach \i in {3,2,...,-3}{
            \filldraw[red] (-\i,0) circle  (0.08);
            \draw[-stealth] (-\i,0) -- (-\i+1,0);
            \draw[-stealth] (-\i,0) .. controls (-\i-0.5,0.7) and (-\i+0.5,0.7) .. (-\i,0);
            \draw[-stealth] (-\i,0) .. controls (-\i-0.8,1) and (-\i+0.8,1) .. (-\i,0);
            \node [below] at (-\i,0){$A_p A_1^{\i}$};
            \node [above] at (-\i+0.5,0){$A_x$};
            \node [above] at (-\i,0.7){$A_p,\,A_1$};
            }
            \end{tikzpicture}
	\caption{Conjugacy graph $\op{sk}_{[A_p]}$}
	\end{center}
	\end{figure}

 \end{example}

\begin{definition}
\label{def-stabilised}
A potential $\phi$ is called \textit{stabilised} to the value $a_0$ at infinity on $\Gamma_{[x_0]}$ if the following condition is satisfied:
\begin{equation}
     \forall \eps > 0\,\, \exists K: \forall g \in \Gamma_{[u_0]} \setminus K \pf |\phi(g) - a_0| < \eps,
\end{equation}
where $K$ is a finite set. This definition only makes sense if the number of vertices in $\Gamma_{[u_0]}$ is infinite.  
\end{definition}


We can adjust the value of a potential by a constant, so in cases where a potential is stabilised, we will consider it to be stabilised to $0.$

In the following sections, it will be shown that for a certain class of groups, potentials of continuous derivations are stabilised. 

The following proposition is obvious:
\begin{proposition}
    Let the elements of $\Gamma_{[u_0]}$ be indexed in some way by $\mathbb{N}$. The potential $\phi$ is stabilised iff the sequence $\{\phi(g_k)\}_{k=1}^\infty$ has a finite limit. 
\end{proposition}

First, let us show  that a potential of derivation $d\colon \C[G] \to \ell_p(G)$ has no sharp changes in potential from point to point as it tends to infinity. This is formulated more precisely in the following proposition. 

\begin{proposition}
   Let $\phi$ be a potential of derivation ${d\colon \C[G] \to \ell_p(G)}$, continuity of $d$ is not assumed. Then for all $\eps > 0,$ there exists a finite set $K\subset G$ such that for all $g_1,\, g_2 \in  G \setminus K$ with $\rho(g_1,g_2) =1,$ we have $|\phi(g_1) - \phi(g_2)| < \eps.$
    \label{prop:weak_cond_1}
\end{proposition}

\begin{proof}
    Otherwise there will be an infinite number of edges within $\op{sk}{G}$ where the potential difference exceeds $\eps.$ 
    So there exists at least one $x\in \X$ such that the set ${\{g \in G : \chi(g,x) > \eps\}}$ is infinite. Consequently, according to the formula $d(g) = \sum_{g \in G} \chi(h,g) g,$ it follows that $d(x_i) \notin \C[G]_{p}.$
\end{proof}

\begin{definition}
\label{def:G-bounded-derivation}
A derivation $d\colon \C[G] \to \cC[G]$ is called \textit{$G$-bounded} if $\sup_{g\in G} \norm{d(g)}$ is finite.  
\end{definition}

It is obvious that continuity implies $G$-boundedness, and in the case of the $\ell_1$ norm the converse is also true (for an operator $A\colon \ell_1 \to X,$ where $X$ is an arbitrary normed space, it holds that $\norm{A} = \sup_{i\in N} \norm{d(e_i)}$). In the general case the converse is false, see appendix \ref{appendix:example-unbounded}.
\section{Bounded Conjugacy Condition}

In contrast to the case of Cayley graphs and the translation action, in the case of the conjugation action we are interested in, the distance between vertices can change uncontrollably.   For example, in a free group $F_2 = \langle a, b | \emptyset \rangle$ we have $\rho(a, bab^{-1}) = 1,$ but $\rho(a^{n}aa^{-n}, a^{n} b a b^{-1} a^{-n}) = n+1,$ i.e. tends to infinity.  For our construction we will introduce a class of groups in which conjugations act in a controllable manner.

\begin{definition}
  \label{def:bounded-conjugations}
A group $G$ is called a group with bounded conjugations (a \textit{BC-group} for short) if the following is satisfied:
    \begin{equation}
         \forall h_1, h_2:   \rho(h_1,h_2) = 1\,\, \exists \, C > 0 : \forall g \in G \, \hookrightarrow \rho(g h_1 g^{-1}, g h_2 g^{-1}) < C.
    \end{equation}
\end{definition}

It is easy to show that $G$ is a BC-group iff for any finite $K\subset G$ holds 
$$\sup_{g\in G}\op{diam}(g^{-1}Kg)<\infty.$$   

\begin{theorem}
The property of a group to be a BC-group is well defined, and invariant under the choice of finite generating set.
\end{theorem}

\begin{proof}
Consider two generating sets $\mathscr{X} = \{x_1, \dots, x_n\}$ and $\mathscr{Y} = \{y_1, \dots, y_m\}.$ Suppose the BC condition holds with respect to $\mathscr{X}$, i.e., for any $h \in G,$ there exists a constant $C = C(B_1(h))$ such that for every $g\in G$ and $x_i \in \mathscr{X},$ we have $\rho_\mathscr{X}(gxhx^{-1} g^{-1}, g h g^{-1}) < C.$

Now express all generators $\mathscr{Y}$ through generators $\mathscr{X}$: 
\begin{equation}
    \begin{split}
        &y_1 = w_1(x_1,\dots, x_n);\\
        &\dots\\
        &y_m = w_m(x_1, \dots, x_n).
    \end{split}
\end{equation}
Let $L_y$ be the maximum length of words $w_1,\dots, w_m.$ Similarly, by expressing $x_i$ through $\mathscr{Y},$ we define $L_x.$  Then, for any $h,g \in G, \,y_i\in \mathscr{Y}$ we obtain
\begin{align}
    \rho_\mathscr{Y}(gyhy^{-1} g^{-1}, g h g^{-1}) \le  L_x \rho_\mathscr{X}(gyhy^{-1} g^{-1}, g h g^{-1}) = \nonumber \\ =  L_x \rho_\mathscr{X}(g w_i(x) h w_i(x)^{-1} g^{-1}, g h g^{-1}) \le L_xC\left(B_{L_y}(h)\right),
\end{align}
which completes the proof.
\end{proof}

\subsection{Examples}
Let us now look at some examples of BC-groups.
\begin{example}
Any nilpotent group of rank 2 is a BC-group because for such groups, the graph $\op{sk}_u$ is isomorphic to the Cayley graph $G/Z(u)$ (see \cite{Aru20-en}, lemma 4). 
The group $G/Z(u)$ is abelian, so the BC condition is satisfied with the constant $C = 1.$ 
\end{example}

In particular the Heisenberg group  
$$H_3(\Z) = \left\{\begin{pmatrix} 1 & a & c \\ 0 & 1 & b \\ 0 & 0 & 1  \end{pmatrix}: a,\,b,\,c \in \Z \right\}$$
is a BC-group.

\begin{example} Each
FC-group is a BC-group, since by definition, all groupoid components in an FC-group are finite. It is known that a finitely generated group $G$ is an FC-group iff the derived subgroup $|G'|$ is finite (see \cite{Neumann-1954}, (3.1) Theorem). 
\end{example}

\begin{example}
    Two previous examples can be generalised as follows: Let $G$ be a group such that $|G'/(Z(G)\cap G')|<\infty,$ then $G$ is a BC-group. Let us show it. 
    Let $\{[a_1],\dots, [a_k]\}$ be the elements of $G'/(Z(G)\cap G'),$ let $A = \{a_1,\dots, a_k\}$ be the set of their representatives. 
    For any $h\in \X,$ $g\in G$ we have $gh = z ahg,$ where $a\in A,$ $z\in Z(G).$ Then we obtain 
    $$\rho(gug^{-1}, gxux^{-1}g^{-1}) = \rho(gug^{-1}, axgu(axg)^{-1}) \leq |a|+1 \leq \max_{a'\in A}|a'| + 1.$$

    Here $|a|$ is a length of $a\in G$ with respect to the generating set $\X,$ that is, the minimum number $n$ such that $a$ can be represented as the product $x_{i_1}\dots x_{i_n}.$ 
\end{example}

\begin{example} 
\label{ex:infinite dihedral group}
The infinite dihedral group $D_{\infty} = \Z_2 * \Z_2 = \langle a ,b \mid a^2 , b^2  \rangle$ is a BC-group. Let us describe the conjugacy classes. If a word $w$ starts and ends with the same letter, then it belongs either to $[a]$ or to $[b].$ If a word $w$ starts and ends with different letters, then it has the form $(ab)^n$ or $(ba)^n.$ For words of this form, we have $[(ab)^n] = [(ba)^n] = \{(ab)^n, (ba)^n\}.$ We are interested in the infinite conjugacy classes. For definiteness, let us work with $[a].$

Let's check the BC condition. Consider an element of the form $(ba)^kb.$ The adjacent elements are $a(ba)^kba$ and $a(ba)^{k-1}.$ Let's see how the distance between the elements $h_1 = (ba)^kb$ and $h_2 = a(ba)^kba$ changes when conjugated by $g.$ Write $g$ as a reduced word and conjugate letter by letter. If $a$ is the rightmost letter in $g,$ then $h_1$ will move 1 step to the right (see Figure 2), and $h_2$ will move 1 step to the left. Then, when conjugating by $b,$ the element $ah_1a$ will move right again, and $ah_2a$ will move left, and so on, until one of them reaches the beginning of the ray. If this happens, they will start moving in the same direction. Therefore, $B_1(h_1) = 2 \rho(h_1,a) + 1.$ The additional 1 comes from the fact that when we reach the end of the segment, i.e., $a,$ we will need to conjugate again with $a,$ so the rightward movement will not start at this step.
 \begin{figure}[H]
	\begin{center}
		\begin{tikzpicture} [font=\small]
		\filldraw[red] (0,1) circle  (0.05); 
		\draw (0,1) -- (1,1);
		\filldraw[red] (1,1) circle  (0.05);
		\draw (1,1) -- (2,1);
		\filldraw[red] (2,1) circle  (0.05); 
		\draw (2,1) -- (3,1);
		\filldraw[red] (3,1) circle  (0.05);
		\draw (3,1) -- (4,1);
		\filldraw[red] (4,1) circle  (0.05); 
		\draw (4,1) -- (5,1);
		\filldraw[red] (5,1) circle  (0.05);
		\draw (5,1) -- (6,1);
		\filldraw[red] (6,1) circle  (0.05); 
		\draw (6,1) -- (7,1);
		\filldraw[red] (7,1) circle  (0.05);
		\coordinate[label = below:$\tiny{a}$] (B) at (0,1);
		\coordinate[label = below:$\tiny{bab}$] (B) at (1,1);
		\coordinate[label = below:$\tiny{ababa}$] (B) at (2,1);
		\end{tikzpicture}
	\caption{Conjugacy graph $\op{sk}_{[a]}$}
	\end{center}
	\end{figure}

\end{example}

\begin{proposition}
    a) Let $G$ and $H$ be a BC-groups, then $G\times H$ is also a BC-group.
    
    b) Let $H$ be a normal subgroup of a BC-group $G.$ Then $G/H$ is a BC-group.  
\end{proposition}

\begin{proof}
     a) Suppose $G$ is generated by $x_1,\dots, x_n,$ and the group $H$ is generated by $y_1,\dots, y_n.$ Then, $G\times H$ is generated by the union of the images of these elements under the embeddings. The BC condition  follows immediately from the following relation on metrics:
     $$\rho_{G\times H}(a_1, a_2) \le \rho_G(\pi_G(a_1), \pi_G(a_2)) + \rho_H(\pi_H(a_1), \pi_H(a_2)), \text{ where } a_1,\,a_2\in G\times H.$$
     The sum of the constants could be taken as the boundedness constant.
     
     b) Let $x_i$ be the generators of the group $G$. Take $\left[x_i\right]$ as the generators of $G / H$. From $\rho\left(\left[g_1\right],\left[g_2\right]\right)=1,$ it follows $\left[x g_1 x^{-1}\right]=\left[g_2\right]$ for some $x\in\mathscr{X}$. Since $G$ is a BC-group, we have $\forall g \in G \hookrightarrow \rho\left(g x g_1 x^{-1} g^{-1}, g g_1 g^{-1}\right)<C$. Therefore, $\rho\left([g]\left[g_1\right]\left[g^{-1}\right],[g]\left[g_2\right]\left[g^{-1}\right]\right)<C$.
\end{proof}

Consequently the finite products of the groups mentioned above and finite groups will satisfy the bounded conjugacy condition.

\begin{example}
    Consider $G = D_\infty \rtimes_{\varphi} \Z_2,$ where $\phi\colon \Z_2 \to \op{Aut}(D_\infty)$ such that $\phi(a) = b,\,\phi(b) = a.$ This group can be presented using generators and relations as
    $G = \langle a, b, c \mid a^2 = b^2 = c^2 = e, cac = b \rangle.$

    Any word in this group can be represented as $u = w(a,b)c^\eps,$ where $\eps = 0,1.$
    If $\eps = 0,$ then $u$ belongs either to the conjugacy class $[a] = [b]$ or to a finite conjugacy class of the form $[(ab)^n].$

    In the case of $\eps = 1,$ the infinite class is $[c] = \{(ab)^n c: n\in \Z\}.$ Finite classes have the form $[(ab)^nac] = \{(ab)^nac, (ba)^n bc \}.$

    Analogously to Example \ref{ex:infinite dihedral group}, it can be shown that the BC condition is satisfied for infinite conjugacy classes. These ones are depicted in the figures below.
\end{example}

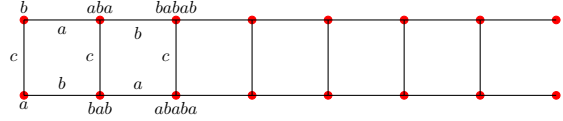
\begin{figure}[H]
	\begin{center}
		\begin{tikzpicture} [font=\small,scale = 1, every node/.style={scale=0.6}]
		\filldraw[red] (0,1) circle  (0.05); 
		\filldraw[red] (1,1) circle  (0.05);
		\filldraw[red] (2,1) circle  (0.05); 
		\filldraw[red] (3,1) circle  (0.05);
		\filldraw[red] (4,1) circle  (0.05); 
		\filldraw[red] (5,1) circle  (0.05);
		\filldraw[red] (6,1) circle  (0.05); 
		\filldraw[red] (7,1) circle  (0.05);
            \filldraw[red] (0,2) circle (0.05);
            \filldraw[red] (1,2) circle  (0.05);
		\filldraw[red] (2,2) circle  (0.05); 
		\filldraw[red] (3,2) circle  (0.05);
		\filldraw[red] (4,2) circle  (0.05); 
		\filldraw[red] (5,2) circle  (0.05);
		\filldraw[red] (6,2) circle  (0.05); 
		\filldraw[red] (7,2) circle  (0.05);
            \draw (0,1) -- (1,1);
            \draw (1,1) -- (2,1);
            \draw (2,1) -- (3,1);
            \draw (3,1) -- (4,1);
            \draw (4,1) -- (5,1);
            \draw (5,1) -- (6,1);
            \draw (6,1) -- (7,1);
            \draw (0,2) -- (1,2);
            \draw (1,2) -- (2,2);
            \draw (2,2) -- (3,2);
            \draw (3,2) -- (4,2);
            \draw (4,2) -- (5,2);
            \draw (5,2) -- (6,2);
            \draw (6,2) -- (7,2);
            \draw (0,1) -- (0,2);
            \draw (1,1) -- (1,2);
            \draw (2,1) -- (2,2);
            \draw (3,1) -- (3,2);
            \draw (4,1) -- (4,2);
            \draw (5,1) -- (5,2);
            \draw (6,1) -- (6,2);
		\coordinate[label = below:$\tiny{a}$] (B) at (0,1);
		\coordinate[label = below:$\tiny{bab}$] (B) at (1,1);
            \coordinate[label = below:$\tiny{ababa}$] (B) at (2,1);
            
            \coordinate[label = above:$\tiny{b}$] (B) at (0,2);
		\coordinate[label = above:$\tiny{aba}$] (B) at (1,2);
            \coordinate[label = above:$\tiny{babab}$] (B) at (2,2);

            \coordinate[label = left:$\tiny{c}$] (B) at (0,1.5);
            \coordinate[label = left:$\tiny{c}$] (B) at (1,1.5);
            \coordinate[label = left:$\tiny{c}$] (B) at (2,1.5);
		\coordinate[label = below:$\tiny{a}$] (B) at (0.5,2);
            \coordinate[label = above:$\tiny{a}$] (B) at (1.5,1);
            \coordinate[label = below:$\tiny{b}$] (B) at (1.5,2);
            \coordinate[label = above:$\tiny{b}$] (B) at (0.5,1);

		\end{tikzpicture}
	\caption{ $\op{sk_a}( D_\infty \rtimes_{\varphi} \Z_2).$}
	\end{center}
	\end{figure}

  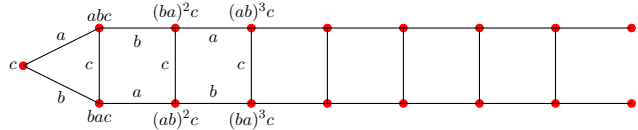
\begin{figure}[H]
	\begin{center}
		\begin{tikzpicture} [font=\small,scale = 1, every node/.style={scale=0.6}]
            \filldraw[red] (-1,1.5) circle  (0.05);
		\filldraw[red] (0,1) circle  (0.05); 
		\filldraw[red] (1,1) circle  (0.05);
		\filldraw[red] (2,1) circle  (0.05); 
		\filldraw[red] (3,1) circle  (0.05);
		\filldraw[red] (4,1) circle  (0.05); 
		\filldraw[red] (5,1) circle  (0.05);
		\filldraw[red] (6,1) circle  (0.05); 
		\filldraw[red] (7,1) circle  (0.05);
            \filldraw[red] (0,2) circle (0.05);
            \filldraw[red] (1,2) circle  (0.05);
		\filldraw[red] (2,2) circle  (0.05); 
		\filldraw[red] (3,2) circle  (0.05);
		\filldraw[red] (4,2) circle  (0.05); 
		\filldraw[red] (5,2) circle  (0.05);
		\filldraw[red] (6,2) circle  (0.05); 
		\filldraw[red] (7,2) circle  (0.05);
            \draw (-1,1.5) -- (0,1);
            \draw (-1, 1.5) -- (0,2);
            \draw (0,1) -- (1,1);
            \draw (1,1) -- (2,1);
            \draw (2,1) -- (3,1);
            \draw (3,1) -- (4,1);
            \draw (4,1) -- (5,1);
            \draw (5,1) -- (6,1);
            \draw (6,1) -- (7,1);
            \draw (0,2) -- (1,2);
            \draw (1,2) -- (2,2);
            \draw (2,2) -- (3,2);
            \draw (3,2) -- (4,2);
            \draw (4,2) -- (5,2);
            \draw (5,2) -- (6,2);
            \draw (6,2) -- (7,2);
            \draw (0,1) -- (0,2);
            \draw (1,1) -- (1,2);
            \draw (2,1) -- (2,2);
            \draw (3,1) -- (3,2);
            \draw (4,1) -- (4,2);
            \draw (5,1) -- (5,2);
            \draw (6,1) -- (6,2);
            \coordinate[label = left: $\tiny{c}$] (B) at (-1,1.5);
		\coordinate[label = below:$\tiny{bac}$] (B) at (0,1);
		\coordinate[label = below:$\tiny{(ab)^2c}$] (B) at (1,1);
            \coordinate[label = below:$\tiny{(ba)^3c}$] (B) at (2,1);
            
            \coordinate[label = above:$\tiny{abc}$] (B) at (0,2);
		\coordinate[label = above:$\tiny{(ba)^2c}$] (B) at (1,2);
            \coordinate[label = above:$\tiny{(ab)^3c}$] (B) at (2,2);

            \coordinate[label = above:$\tiny{a}$] (B) at (-0.5,1.75);
            \coordinate[label = below:$\tiny{b}$] (B) at (-0.5,1.25);
            \coordinate[label = left:$\tiny{c}$] (B) at (0,1.5);
            \coordinate[label = left:$\tiny{c}$] (B) at (1,1.5);
            \coordinate[label = left:$\tiny{c}$] (B) at (2,1.5);
		\coordinate[label = below:$\tiny{b}$] (B) at (0.5,2);
            \coordinate[label = above:$\tiny{b}$] (B) at (1.5,1);
            \coordinate[label = below:$\tiny{a}$] (B) at (1.5,2);
            \coordinate[label = above:$\tiny{a}$] (B) at (0.5,1);

		\end{tikzpicture}
	\caption{$\op{sk}_c( D_\infty \rtimes_{\varphi} \Z_2.)$}
	\end{center}
	\end{figure}

 \begin{example}[Not a BC-group]
Consider $G = H_3 \rtimes_{\varphi} \Z_2$, where $\phi\colon \Z_2 \to \op{Aut}(H_3)$ is defined by $\phi(A_p) = A_x,\,\phi(A_x) = A_p,\, \phi(A_1) = A_1^{-1}.$ This indeed defines a homomorphism, as the relations go to relations. Indeed, $[\phi(A_p), \phi(A_x)] = [A_x, A_p] = [A_p, A_x]^{-1} = A_1^{-1} = \phi(A_1),$ and the remaining two are obvious.

 Now notice that $cA_xc = A_p,$ but $\rho(A_x^k A_p A_x^{-k}, A_p)\to \infty,$ whereas  $\rho(A_x^k A_x A_x^{-k}, A_x) = 0.$ So BC condition isn't satisfied.

 \begin{figure}[H]
	\begin{center}
		\begin{tikzpicture} [font=\small,scale = 1, every node/.style={scale=0.6}]
            \foreach \i in {-4,-3,...,4}{
            \filldraw[red] (\i,1) circle  (0.05);
            \filldraw[red] (\i,2) circle  (0.05);
            \draw (\i,1) -- (\i+1,1);
            \draw (\i,2) -- (\i+1, 2);
            \draw (\i,1) -- (\i,2);
            }
            \draw (-5,1) -- (-4,1);
            \draw (-5,2) -- (-4,2);
		\coordinate[label = below:$\tiny{A_p}$] (B) at (0,1);
		\coordinate[label = below:$\tiny{A_p A_1^{-1}}$] (B) at (1,1);
            \coordinate[label = below:$\tiny{A_pA_1^{-2}}$] (B) at (2,1);
            
            \coordinate[label = above:$\tiny{A_x}$] (B) at (0,2);
		\coordinate[label = above:$\tiny{A_x A_1}$] (B) at (1,2);
            \coordinate[label = above:$\tiny{A_xA_1^2}$] (B) at (2,2);

            \coordinate[label = left:$\tiny{c}$] (B) at (0,1.5);
            \coordinate[label = left:$\tiny{c}$] (B) at (1,1.5);
            \coordinate[label = left:$\tiny{c}$] (B) at (2,1.5);
		\coordinate[label = below:$\tiny{A_p}$] (B) at (0.5,2);
            \coordinate[label = below:$\tiny{A_p}$] (B) at (1.5,2);
            \coordinate[label = above:$\tiny{A_x}$] (B) at (1.5,1);
            \coordinate[label = above:$\tiny{A_x}$] (B) at (0.5,1);
            \node [above] at (-1,2) {$A_xA_1^{-1}$};
            \node [below] at (-1,1) {$A_pA_1$};
            \node[right] at (-1,1.5) {$c$};
            
		\end{tikzpicture}
	\caption{$\op{sk}_{A_x}(H_3\rtimes_{\varphi} \Z_2)$}
	\end{center}
	\end{figure}
\end{example}

This example also demonstrates that not all quasi-isometries preserve the BC condition, so the property of a group to be a BC group is not a coarse invariant. Also it shows that on conjugacy classes with two ends the BC-condition can be not satisfied.

\section{Derivations in BC-groups}
In this section we will prove that in BC-groups with uniformly bounded finite conjugacy classes all continuous derivations are inner. First we will show that on each infinite component a potential corresponding to the derivation is stabilised (see Definition \ref{def-stabilised}). Then we will prove the theorem for derivations with support in one infinite component, and finally we will prove the initial statement.

\begin{lemma}
    Let $G = \langle x_1, \dots , x_n | r_i \rangle $ satisfy the BC condition.
    Let $\phi$ be a potential corresponding to $G$-bounded derivation ${d \colon f_p(G) \to \ell_q(G)}$. Then $\phi$ is stabilised. 
    \label{lem:even-out}
\end{lemma}

\begin{proof}
    Consider an arbitrary infinite connected component $\op{sk}_{g_0}$. Define numbers
    \begin{equation}
        a := \inf_K \left(\sup\limits_{g \in \Gamma [g_0] \setminus K } \phi\left(g\right)\right), \quad 
        b := \sup_K \left(\inf \limits_{g \in \Gamma [g_0] \setminus K } \phi\left(g\right)\right),\quad \delta:= a-b.
    \end{equation}
    

    Clearly, if the supremum or infimum equals infinity, the derivation is not continuous. 
    
    Suppose $\phi$ isn't stabilised; then $\delta$ is positive. Fix an arbitrary natural number $n$. Consider unbounded sets $$V_a 
    = \{g \mid |\phi(g)-a| < \frac{\delta}{4}\},\quad V_b = \{g \mid |\phi(g)-b| < \frac{\delta}{4}\}.$$ 
    
    Take an arbitrary $n$-element subset $V_n \subset V_a.$ Applying the BC condition we find a constant $C$ such that for all $h_1, h_2 \in V_n$ and all $g \in G,$ the inequality $\rho(gh_1 g^{-1}, gh_2 g^{-1}) < C$ is satisfied. 
    According to Proposition \ref{prop:weak_cond_1}, there exists $h \in V_b$ such that $B_C(h) \subset V_b.$ Consider an arbitrary $u \in V_n.$ Since $u,\, h$ belong to the same connected component, we have $h = g u g^{-1}.$ Thus, $g V_n g^{-1} \subset B_C(h) \subset V_b.$
    
    Using the formula \ref{eq:d(g)-phi}, we obtain
    \begin{align}
        \norm{d(g)} = \norm{\sum_{t\in G} \left(\phi(gtg^{-1})  - \phi(t)\right)gt} =   
        \ge \sqrt[q]{\sum_{t\in V_n} \left|\phi(gtg^{-1})  - \phi(t)\right|^{q}} \ge \sqrt[q]{n\frac{\delta}{2}}.
    \end{align}

    Since $n$ was chosen arbitrarily, we conclude that $d$ is unbounded. 
\end{proof}


\subsection{Case of a Single Component}

\begin{lemma}
    Let $G$ be a BC-group. Then any $G$-bounded derivation with support in one infinite component $\Gamma_{[u_0]}$ is inner. 
        \label{lem:potential-in-lq-in-one-component}
\end{lemma}
\begin{proof}
    Due to equation \eqref{eq:formal-commutator}, it is sufficient to prove that a potential from $\ell_q$ can be chosen for the given derivation. Suppose the opposite:
    \begin{equation}
        \sum_{g \in \Gamma [u_0]} |\phi (g)|^q = \infty. 
    \end{equation}

    For each $x_0 > 0$, there exist $g_1, \dots, g_n\in \op{Obj}(\Gamma_{[u_0]})$ such that 
    $${|\phi (g_1)|^q + \dots + |\phi (g_n)|^q > x_0^q},$$
    and each term in the sum is nonzero. 
    Set $m = \frac{1}{2}\min \{|\phi (g_1)|, \dots |\phi(g_n)|\}.$  According to Lemma \ref{lem:even-out}, the potential converges to $0$ at infinity. So we can find $R > 0$ such that $|\phi(g)| < m$ for all $g \in \Gamma [u_0] \setminus B_R(u_0)$.  By the BC condition there exists a constant $C = C(B_R(u_0)) > 0$ such that for all $g \in G$ and all $h_1, h_2 \in B_R(u_0)$, we have $\rho(g h_1 g^{-1}, g h_2 g^{-1}) < C.$
    
    Consider an element $g \in G$ 
    such that $\rho(u_0, g u_0 g^{-1}) > R + C$. Then
    \begin{equation}
    \forall h \in B_R (u_0) \hookrightarrow |\phi(g h g^{-1})| < m,
    \end{equation}
    because $\rho(ghg^{-1}, gu_0g^{-1})<C,$ and therefore $ghg^{-1}\notin B_R(u_0).$ 
    
    From (\ref{eq:d(g)-phi}) we obtain lower bound of $\norm{d(g)}$:
    \begin{align}
        \norm{d(g)} = \norm{\sum_{t\in G} \left(\phi(gtg^{-1})  - \phi(t)\right)gt} = \sqrt[q]{\sum_{t\in G} \left|\phi(gtg^{-1})  - \phi(t)\right|^{q}} \ge \nonumber \\ 
        \ge \sqrt[q]{\sum_{t\in \{g_1,\dots, g_n\}} \left|\phi(gtg^{-1})  - \phi(t)\right|^{q}} = \sqrt[q]{ \sum_{i = 1}^n \left|\phi(g_i) - \phi(g g_i g^{-1})\right|^q} \ge \nonumber \\ \ge \sqrt[q]{\sum_{i = 1}^m \left|\left|\phi(g_i)\right| - m \right|} 
        \ge \frac{1}{2} \sqrt[q]{\sum_{i = 1}^n |\phi(g_i)|^q} = \frac{x_0}{2}. 
    \end{align}
    The first inequality holds because we just omitted all summands with $t\notin \{g_1,\dots,g_n\},$ after that we applied the triangle inequality and took into account that $\phi(g g_i g^{-1}) < m < \min_{i=1,\dots,n}|\phi(g_i)|. $

    Thus, for any chosen $x_0$, we have found an element $g\in G$ for which $\norm{d(g)}\ge \frac{x_0}{2}.$ This implies that $d$ is not $G$-bounded.
\end{proof}

Therefore, for BC-groups, it is shown that a derivation with support in one component can be associated with a potential from $\ell_q.$ 

\subsection{Case of Multiple Components}
\label{subsec:Case of Multiple Components}

Let us examine how the distances $\rho(h,  g h g^{-1})$ and $\rho(h, g^{-1}hg)$ are related. Consider $g = x^k y,\,h = x \in F_2 = \langle x,y \rangle.$ Then
$$\rho(h, ghg^{-1}) = \rho(x, x^{k}yxy^{-1}x^{-k}) = k+1,$$
$$\rho(h, g^{-1}hg) = \rho(x,y^{-1}xy) = 1,$$
so in general case one of the distances can increase largely, whereas other is bounded. But if $G$ is a BC-group, it is impossible.

\begin{proposition}
    Let $G$ be a BC-group, and let $\{a_k\}_{k=1}^\infty$ be a sequence of elements of $G$ such that ${\rho(u, a_k u a_k^{-1}) \to \infty.}$ Then ${\rho(u, a_k^{-1} u a_k) \to \infty.}$ 
    
    \label{prop:inverse-sequence}
\end{proposition}

\begin{proof}
    Otherwise, there exists a constant $L$ and a subsequence $a_{n_k}$ such that $\rho(u, a_{n_k}^{-1} u a_{n_k}) \le L,$ meaning $a_{n_k}^{-1}u a_{n_k}\in B_L(u).$ Conjugating by $a_{n_k}$ and using the boundedness of conjugations, we obtain $\rho(a_{n_k}u a_{n_k}^{-1},u) < C(B_{L}(u))$, which contradicts the hypothesis that ${\rho(u, a_k u a_k^{-1}) \to \infty.}$ 
\end{proof}


\begin{lemma}
    Let $G$ be a BC-group. There exists a sequence $\{a_k\}_{k=1}^{\infty},$  $a_k\in G$ such that for every $h$ from any infinite connected component we have $\rho(h, a_k h a_k^{-1})\to \infty.$ 
    \label{lem:sequence-leading-to-infty}
\end{lemma}
\begin{proof}
    First, consider the finite number of components. Let us enumerate all infinite connected components and proceed by induction on the number of components. In each component, we arbitrarily choose an origin $u_i.$ 

    The base case is evident; the existence of such a sequence for a single component follows from the fact that we are dealing with infinite components.
    
    Suppose $a_n$ is a sequence such that $\rho(u_i, a_n u_i a_n^{-1}) \to \infty$ for $i \le k.$ Consider the component $\Gamma_{[u_{k+1}]}.$ If we can extract a subsequence from the sequence $\rho(u_{k+1}, a_n u_{k+1} a_n^{-1})$ that converges to infinity, this completes the induction step.
    
    Otherwise, $\{a_n h a_n^{-1}| n\in \mathbb{Z}\}$ will be bounded for each $h\in \Gamma_{[u_{k+1}]}.$ Take an arbitrary sequence $b_n$ such that $\rho(u_{k+1}, b_n u_{k+1} b_n^{-1}) \to \infty.$ Extract a subsequence $a_n'$ from the sequence $a_n$ such that $\min_{i \le k}\rho(u_i, a_n' u_i {a_n'}^{-1}) > 2 l(b_n).$ Consider the sequence $c_n = a_n' b_n.$ Then, for each $i \le k+1,$ we have $\rho(u_i, c_n u_i c_n^{-1}) \to \infty.$ For $i \le k,$ this follows from the triangle inequality. For $i = k+1,$ it follows from the fact that $\{a_n u_{k+1} a_n^{-1}| n\in \mathbb{Z}\}$ is a bounded set, $G$ is a BC-group and   $\rho(u_{k+1}, b_n u_{k+1} b_n^{-1}) \to \infty.$

    In case of an arbitrary number of components, applying the result above, we can choose an element $a_k \in G$ such that $\min_{i \in \mathbb{N}} \rho(u_i, a_k u_i a_k^{-1}) > k.$ 
    Then, for every $i \in \mathbb{N},$ it holds that $\rho(u_i, a_k u_i a_k^{-1}) \to \infty,$ and thus, for every $h \in \Gamma_{[u_i]},$ $\rho(u_i, a_k h a_k^{-1}) \to \infty.$
    
\end{proof}

Let $\Gamma_{\text{inf}}$ and $\Gamma_{\op{f}}$ denote the union of all infinite and all finite components, respectively.

The result of Lemma $\ref{lem:potential-in-lq-in-one-component}$ can be generalised, it is shown in Lemma \ref{lem:potential-in-lq}. 

Define the $\ell_q$ norm of the potential $\phi$ by the formula $$\norm{\phi}_q = \sqrt[q]{\sum_{g\in G} |\phi(g)|^q}.$$ 

\begin{lemma}
    Let the support of the derivation $d$   be contained in $\Gamma_{\op{inf}}$. Suppose a potential $\phi$ of $d$ stabilises to zero. Let $\{a_k\}_{k = 1}^{\infty}$ be a sequence such that for each $u \in \op{\Gamma_{\op{inf}}},$ $\rho(u, a_k u a_k^{-1})\to \infty.$
    Then ${\lim_{k\to \infty}\norm{d(a_k)} = \sqrt[q]{2}\norm{\phi}_q.}$  
    \label{lem:potential-in-lq}
\end{lemma}

\begin{proof}

In each infinite connected component choose origin $u_i$. Pick $0<\eps<1.$ 

\textbf{Case 1:} $0<\norm{\phi}_q<\infty$ (the case of zero norm is trivial).  

Choose $x_0 = \norm{\phi}_q - \eps/2$. Then, there 
exists a finite set  $\mc{S} = \{g_1, \dots, g_n\}\subset  \op{Obj}(\Gamma_{\op{inf}}),$ 
such that 
    $$\sum_{g\in \mc{S}}|\phi (g)|^q > x_0^q, \quad
    \sum_{g \notin \mc{S}} |\phi(g)|^q < \min \left( \frac{\eps}{2},\,\, 2^{-q}\eps \left(\norm{\phi}_q\right)^q  \right),$$
   and $\phi(g)\neq 0$ for any $g \in \mc{S}$. We will need the first part of the minimum for the lower bound estimate and the second part for the upper bound estimate. 
    
     Since $ \mathcal{S}$ is a finite set it is contained in a finite union of infinite connected components, that is, $\mc{S}\subset \cup_{i \in I} \Gamma_{[u_i]},\,|I|<\infty.$  So, there exists $r>0$ such that $\{g_1,\dots, g_n\} \subset \cup_{i\in I} B_r(u_i)$.

     By hypothesis, the potential stabilises to zero at infinity on each component. Therefore, for any $m>0$ there exists $R>r>0$ such that for each $g \in \cup_{i \in I}\left(\Gamma [u_i] \setminus B_R(u_i)\right),$ we have $|\phi(g)| < m$. Set $m = \frac{\eps}{2}\min \{|\phi (g_1)|, \dots |\phi(g_n)|\}$. Since $G$ is a BC-group, the numbers $C_i = \sup_{g\in G}{\op{diam}(g B_r(u_i)g^{-1})}$ are finite. So the constant ${C = \max_{i \in I} C(B_r(u_i))}$ is finite and satisfies  
    $$\forall g \in G,\,\, \forall i \in I, \forall h_1, h_2 \in B_r (u_i) \hookrightarrow \rho(g h_1 g^{-1}, g h_2 g^{-1}) < C.$$
    
    By hypothesis and Proposition $\ref{prop:inverse-sequence}$, there exists a number $ N$ such that for each $k \ge N$, for each $i \in I,$ we have $\rho(u_i, a_k u_i a_k^{-1}) > R + C$ and $\rho(u_i, a_k^{-1} u_i a_k) > R + C,$ remembering that $I$ is finite.   Then for all $h \in B_r (u_i),$ $k\ge N$ it follows that
    \begin{equation}
    |\phi(a_k h a_k^{- 1})| < m,\quad |\phi(a_k^{-1} h a_k)| < m.
    \end{equation}


  Note that $\mc{S} \cap \left( a_k^{-1}\mc{S}a_k \right)= \emptyset$ by the triangle inequality.  
    
    Now from (\ref{eq:d(g)-phi}), we obtain lower bound: 
    \begin{align}
        \left(\norm{d(a_k)}\right)^{q} = \norm{\sum_{t\in G} \left(\phi(a_kta_k^{-1})  - \phi(t)\right)a_kt}^{q} = {\sum_{t\in G} \left|\phi(a_kta_k^{-1})  - \phi(t)\right|^{q}} \ge \nonumber \\
        \ge {\sum_{t\in \mc{S}} \left|\phi(a_kta_k^{-1})  - \phi(t)\right|^{q}} +  {\sum_{t\in a_k^{-1}\mc{S}a_k} \left|\phi(a_kta_k^{-1})  - \phi(t)\right|^{q}} = \nonumber \\
        = { \sum_{i = 1}^n \left| \phi(a_k g_i a_k^{-1}) - \phi(g_i) \right|^q}  +  { \sum_{i = 1}^n \left|\phi(g_i) - \phi(a_k^{-1}g_ia_k) \right|^q}\ge \nonumber \\ \ge 2{\sum_{i = 1}^m \left|\left|\phi(g_i)\right| - m \right|^{q}} 
        \ge 2(1-\frac{\eps}{2})^q \sum_{i=1}^n|\phi(g_i)|^q  \ge 2\left(\norm{\phi}_q\right)^q\left(1-\frac{\eps}{2}\right)^q\left(1-\frac{\eps}{2}\right).
    \end{align}
    
    This means that $$\norm{d(a_k)} \ge \sqrt[q]{2}\norm{\phi}_q (1-\frac{\eps}{2})\sqrt[q]{1-\frac{\eps}{2}} \ge (1-\eps) \sqrt[q]{2}\norm{\phi}_q.$$
    
    Now, let us obtain upper bound:
    \begin{align}
        &\left(\norm{d(a_k)}\right)^q = {\sum_{t\in G} \left|\phi(a_kta_k^{-1})  - \phi(t)\right|^{q}} \le \nonumber \\
        &{\sum_{t\in \mc{S}} \left|\phi(a_kta_k^{-1})  - \phi(t)\right|^{q}} 
        + {\sum_{t\in a_k^{-1}\mc{S}a_k} \left|\phi(a_kta_k^{-1})  - \phi(t)\right|^{q}} +  {\sum_{t\notin \mc{S}\cup a_k^{-1}\mc{S}a_k} \left|\phi(a_kta_k^{-1})  - \phi(t)\right|^{q}} \le \nonumber \\
        &2\left(1+\frac{\eps}{2}\right)^q \left(\norm{\phi}_q\right)^q  
        +  {\sum_{t\notin \mc{S}\cup a_k^{-1}\mc{S}a_k} 2^{q-1} \left( \left|\phi(a_kta_k^{-1}) \right|^q+ \left|\phi(t)\right|^{q}\right)}  \le 2(1+\frac{\eps}{2})^q \norm{\phi}_q^q + \frac{\eps}{2}\norm{\phi}_q^q.
    \end{align} 

    To estimate the term ${\sum_{t\notin \mc{S}\cup a_k^{-1}\mc{S}a_k} \left|\phi(a_kta_k^{-1})  - \phi(t)\right|^{q}}$ we used an inequality $$(|a|+|b|)^q \le 2^{q-1}(|a|^q + |b|^q),$$ 
    and after that applied $\sum_{g \notin \mc{S}} |\phi(g)|^q <  2^{-q}\eps \left(\norm{\phi}_q\right)^q.$
    
    Thus, 
    $$\norm{d(a_k)} \le \left(1+\frac{\eps}{2}\right)\sqrt[q]{2}\norm{\phi}\sqrt[q]{1 + \frac{\eps}{4}} \le (1+\eps)\sqrt[q]{2}\norm{\phi}.
    $$
    This gives us the desired result: 
    \begin{equation}
        \lim_{k\to \infty} \norm{d(a_k)} = \sqrt[q]{2}\norm{\phi}_q.
    \end{equation}

\textbf{Case 2:} $\norm{\phi}_q = \infty$.

Choose an arbitrary $x_0 > 0$. Then, there exist $g_1, \dots, g_n\in \cup_{i \in I} \op{Obj}(\Gamma_{[u_i]})$ such that 
    $${|\phi (g_1)|^q + \dots + |\phi (g_n)|^q > x_0^q}.$$
    
Analogously to how we calculate the lower bound, for large enough $k$ we obtain 
$$\left(\norm{d(a_k)}\right)^{q} \ge 2{\sum_{i = 1}^m \left|\left|\phi(g_i)\right| - m \right|^{q}} \ge 2\left(1 - \frac{\eps}{2}\right) x_0.$$

Since $x_0$ could be chosen arbitrarily large, we conclude that sequence converges to infinity. 
\end{proof}


\begin{corollary}
    Let $G$ be a BC-group, and let ${d \colon f_p(G) \to \ell_q(G)}$ be a bounded derivation with support in $\Gamma_{\text{inf}}$. Then $d$ has a potential from $\ell_q.$
    \label{cor:potential-in-lq-infinite-components}
\end{corollary}
\begin{proof}
Any potential stabilises by lemma \ref{lem:even-out}. Choose such $\phi$ that stabilises to $0$. Consider a sequence $a_k$ from lemma \ref{lem:sequence-leading-to-infty}. By lemma \ref{lem:potential-in-lq} we get  
$$\norm{\phi}_q = \lim_{i\to \infty}\norm{d(a_{i})} \le \norm{d}_{el},$$
where  $\norm{d}_{el} = \sup_{g \in G} \norm{d(g)}.$
\end{proof}

These results allow us to understand what happens on infinite components. It turns out that if the sizes of finite components are uniformly bounded, then a similar result holds for them.

\begin{proposition}
    Let the diameters of all finite connected components of $\op{sk}_G$ be uniformly bounded by some constant $N$. Then a $G$-bounded derivation $d\colon f_p(G)\to \ell_q(G)$, whose support is contained in $\Gamma_{\op{f}}$, has a potential from $\ell_q.$
    \label{prop:potential-in-lq-finite-component}
\end{proposition}
\begin{proof}
In each finite component, pick a vertex $u_i$ and set the potential  at that vertex to zero. Consider a new generating system for the group $G$, containing all words of length at most $N$. Denote this system by $\mc{A}$.

Then for each $g \in \Gamma_{\op{f}}$, there exists a vertex $u_i$ (a vertex of the component containing $g$) and a generating element $a$ such that $g = a^{-1} u_i a.$

Therefore,

$$ \sum_{g \in \Gamma_{\op{f}}} |\phi(g)|^{q} \le \sum_{a\in \mc{A}} \sum_{u_i \in \Gamma_{\op{f}}} |\phi(a^{-1} u_i a )|^q \le \sum_{a \in \mc{A}} \norm{d(a)} \le |\mc{A}|\norm{d}_{el}. $$

So the potential belongs to $\ell_q.$
\end{proof}

\begin{theorem}
\label{th:main}
    Let $G$ be a BC-group with uniformly bounded finite components. Then every derivation $G$-bounded derivation $d\colon f_p(G)\to \ell_q(G)$ is inner.
\end{theorem}
\begin{proof}
Let's show that for such $d$, we can find a potential from $\ell_q.$ The derivation $d$ can be represented as the sum $d = d_{\text{inf}} + d_{\op{f}}.$ The support of $d_{\text{inf}}$ is contained in the infinite components, and the support of $d_{\op{f}}$ is contained in the finite components. Note that for every $g\in G$, it holds $\norm{d_{\text{inf}}(g)}\le \norm{d(g)} \le \norm{d}_{el},$ similarly for the second term. This means that $d_{\op{f}}$ and $d_{\text{inf}}$ are $G$-bounded, and therefore, by Proposition \ref{prop:potential-in-lq-finite-component} and Corollary \ref{cor:potential-in-lq-infinite-components} their potentials lie in $\ell_q.$ Thus, we obtain
\begin{equation}
    \sum_{g \in G} |\phi(g)|^{q} = \sum_{g \in \Gamma_{\text{inf}}} |\phi(g)|^q + \sum_{g \in \Gamma_{\op{f}}} |\phi(g)|^q = \sum_{g \in G} |\phi_{\text{inf}}(g)|^q + \sum_{g \in G} |\phi_{\op{f}}(g)|^q < \infty. 
\end{equation}

Thus, we have shown $\phi$ lies in $\ell_q.$
\end{proof}

\begin{example}
Nilpotent groups of rank 2 are BC-groups with uniformly bounded finite components.
\end{example}

Recall that a group $G$ is called a nilpotent group of rank 2 if its factor group modulo the center $Z$ is commutative.

Choose a presentation of $G$ of the following form $\langle a_1, a_2, \dots a_m, b_{ij} \dots| R \rangle,$ where $b_{ij} = [a_i^{-1},a_j^{-1}]$. Additionally, assume that if $a_i$ is a generator, then $a_i^{-1}$ is also a generator, i.e., there exists $j$ such that $a_j = {a_i}^{-1}.$ We consider the group to be finitely generated, as always. We denote the center of the group by $Z,$ and the subgroup generated by the commutators of generators by $B = \langle b_{ij} \rangle \subset Z.$

\begin{proposition}
    In a nilpotent group of rank 2, the bounded connected components are uniformly bounded by some constant $N.$
\end{proposition}

\begin{proof}
The group $G' = [G,G]$ is finitely generated. Indeed, for a nilpotent group of rank 2, it holds that $a_i a_j = a_j a_i a_i^{-1} a_j^{-1} a_i a_j = a_j a_i [a_i^{-1}, a_j^{-1}] = a_j a_i b_{ij},$ where $b_{ij}\in Z.$ Consequently, the commutator of any two elements is equal to the product of elements $b_{ij},$ and there is a finite number of $b_{ij}$.

In a finitely generated abelian group, the torsion subgroup is finite. Consider a finite component $\op{sk}_{u_0}$. Notice that $g u_0 g^{-1} = u_0 b,$ where $b\in \op{Tor}(G').$ Therefore, the number of elements in each component is bounded by $N = |\op{Tor}(G')|.$
\end{proof}

\begin{corollary}
In a nilpotent group of rank 2, every $G$-bounded ${d \colon f_p(G) \to \ell_q(G)}$, has a potential from $\ell_q,$ so $d$ is inner.
\end{corollary}

\appendix

\section{Unbounded inner derivation}
\label{appendix:example-unbounded}

Consider the following matrices 
\begin{equation}
    A_p = \begin{pmatrix} 1 & 1& 0 \\ 0 & 1 & 0 \\ 0 & 0 & 1  \end{pmatrix},\quad A_x  = \begin{pmatrix} 1 & 0& 0 \\ 0 & 1 & 1 \\ 0 & 0 & 1  \end{pmatrix}, \quad A_1 = \begin{pmatrix} 1 & 0& 1 \\ 0 & 1 & 0 \\ 0 & 0 & 1  \end{pmatrix}.
\end{equation}

A matrix with elements $(a, b, c)$ can be represented as the product $A_x^b A_p^a A_1^c$. Taking into account the relations, we have $H_3(\Z) = \langle A_x, A_p, A_1 \mid [A_p, A_x] = A_1, [A_p, A_1] = E, [A_x, A_1] = E \rangle.$ Let's consider the connected component containing $A_x^b A_p^a A_1^c.$ Evidently $a$ and $b$ are invariant under conjugation and $c$ can change to values multiples of $a$ and $b.$ The distance between elements is determined only by the difference in $c$, so it does not change during conjugation.

The property of derivation $d$ to be $G$-bounded is weaker than to be continuous. So unbounded inner derivations can exist.

Indeed, consider the Heisenberg group $G = H_3(\Z).$ We will examine derivations of the form $d\colon \C[G]_2 \to \C[G]_2.$


Let's take $h = A_p$ and define a derivation in such a way that there is a non-zero coefficient before $h$ in $d(A_x^k)$. To achieve this, define the potential as $\frac{1}{k}$ on the vertices $A_pA_x^{-k} = A_x^{-k}A_pA_1^{-k}$ for $k \ge 1$ and set it to zero on the remaining vertices.

Derivation $d$ defined by $\phi$ is obviously inner and $G$-bounded since $\phi$ lies in $\ell_2.$ Now let us calculate the images of some elements.

\begin{align}
    d(A_x) &= \sum_{t\in G} \left(\phi(A_x t A_x^{-1})  - \phi(t)\right)A_x t = \sum_{t \in G} \phi(t)tA_x - \sum_{t\in G}\phi(t)A_xt = \nonumber \\
    &=\sum_{k=1}^\infty \frac{1}{k}\left(   A_pA_x^{-k+1} - A_xA_pA_x^{-k} \right) = \sum_{k=1}^\infty \frac{1}{k}\left(   A_x^{-k+1}A_pA_1^{-k+1} - A_x^{-k+1}A_pA_1^{-k} \right)\\
    d(A_x^m) &=  \sum_{k=1}^\infty \frac{1}{k}\left(   A_x^{-k+m}A_pA_1^{-k+m} - A_x^{-k+m}A_pA_1^{-k} \right))
\end{align}

Consider the elements $a_m = \sum_{k = -m}^{m} A_x^k.$ The coefficient in $d(a_m)$ before $A_x^{-n}A_pA_1^{-n}$ is:

\begin{equation}
    \delta_{A_x^{-n}A_pA_1^{-n}}(d(a_m)) = \sum_{k = -m}^m \delta_{A_x^{-n}A_pA_1^{-n}}(d(A_x^k)) = \sum_{k =\max\{ -n+1,-m\},\,k\neq 0}^m \frac{1}{k+n} = \sum_{j=1,\,j\neq n}^{m+n}\frac{1}{j} \ge \sum_{j=2}^m \frac{1}{j},
\end{equation}
the last equality holds only when $n \le m + 1.$

Now we can estimate the norm of the image:

\begin{align*}
    \norm{d(a_m)} \ge \left( \sum_{n = 0}^m |\delta_{A_x^{-n}A_pA_1^{-n}}(d(a_m))|^2  \right)^\frac{1}{2} \ge \sqrt{m} \left(\sum_{j=2}^m \frac{1}{j} \right),
\end{align*}

where $\norm{a_m} = \sqrt{2m+1}.$ Thus, $\frac{\norm{d(a_m)}}{\norm{a_m}} \to \infty,$ so the operator is unbounded.

\newpage

\printbibliography[
heading=bibintoc,
title={References}
]

\end{document}